\documentclass[12pt,reqno]{article}

\usepackage[usenames]{color}
\usepackage{amssymb}
\usepackage{graphicx}
\usepackage{pstricks}
\usepackage{enumitem}

\usepackage{tkz-euclide}
\usepackage{tikz-cd}
\usepackage{mathtools}

\usepackage[colorlinks=true,
linkcolor=webgreen,
filecolor=webbrown,
citecolor=webgreen]{hyperref}

\definecolor{webgreen}{rgb}{0,.5,0}
\definecolor{webbrown}{rgb}{.6,0,0}

\usepackage{color}
\usepackage{fullpage}
\usepackage{float}

\usepackage{graphics,amsmath,amssymb}
\usepackage{amsthm}
\usepackage{amsfonts}
\usepackage{latexsym}
\usepackage{epsf}

\begin{document}

\begin{center}
\epsfxsize=4in
\end{center}

\theoremstyle{plain}
\newtheorem{theorem}{Theorem}
\newtheorem{corollary}[theorem]{Corollary}
\newtheorem{lemma}[theorem]{Lemma}
\newtheorem{proposition}[theorem]{Proposition}

\theoremstyle{definition}
\newtheorem{definition}[theorem]{Definition}
\newtheorem{example}[theorem]{Example}
\newtheorem{conjecture}[theorem]{Conjecture}

\theoremstyle{remark}
\newtheorem{remark}[theorem]{Remark}

\begin{center}
\vskip 1cm{\LARGE\bf Elementary number-theoretical statements proved by Language Theory} \vskip 1cm
\large Jos\'e Manuel Rodr\'iguez Caballero\\
D\'epartement de Math\'ematiques\\
UQ\`AM\\
Case Postale 8888, Succ. Centre-ville  \\
Montr\'eal, Qu\'ebec H3C 3P8
Canada\\
\href{mailto:rodriguez\_caballero.jose\_manuel@uqam.ca}{\tt rodriguez\_caballero.jose\_manuel@uqam.ca}\\
\end{center}

\def\DD{\mathcal{D}}
\def\SS{\mathcal{S}}

\def\blocks{\textrm{blocks}}

\def\ct{\textrm{ct}}

\def\llangle{\langle\!\langle}
\def\rrangle{\rangle\!\rangle}

\vskip .2 in
\begin{abstract}
We introduce a method to derive theorems from Elementary Number Theory by means of relationships among formal languages. Using $\sigma$-algebras, we define what \emph{a proof of a number-theoretical statement by Language Theory} means. We prove that such a proof can be transformed into a traditional proof in $\textbf{ZFC}$. Finally, we show some examples of non-trivial number-theoretical theorems that can be proved by formal languages in a natural way. These number-theoretical results concern densely divisible numbers, semi-perimeters of Pythagorean triangles, middle divisors and partitions into consecutive parts.
\end{abstract}

\section{Introduction}

It is a rather subjective matter to decide whether a given statement in \textbf{ZFC} belongs to the field of Elementary Number Theory or not. A typical example is Goodstein's Theorem, which, even if it concerns positive integers, it has been traditionally classified as belonging to the field of Symbolic Logic (see \cite{Goodstein}). 

Throughout this paper, we will be interested in theorems of the form ``$R = S$'' in \textbf{ZFC}, where $R$ and $S$ are subsets of the set of positive integers, denoted $\mathbb{Z}_{\geq 1}$. Beside the above-mentioned remark, we will say, in a rather informal way, that ``$R = S$'' is an \emph{elementary number-theoretical statement} if $R$ and $S$ concern some kind of integers traditionally studied in Elementary Number Theory, e.g. prime numbers, perfect numbers, square free numbers, integers which are the sum of two squares, etc. We will leave open the question of what is not an elementary number-theoretical theorem.

Our standpoint is to assign a word $\gamma(n) \in \Sigma^{\ast}$ over a finite alphabet $\Sigma$ to any $n \in \mathcal{U}$, where $\mathcal{U}$ is a subset of $\mathbb{Z}_{\geq 1}$. The traditional way to do it is by means of the decimal positional numeration system, where $\Sigma = [0..9]$ and $\mathcal{U} = \mathbb{Z}_{\geq 1}$. In this case, $\gamma^{-1}(w)$ is either the empty set (g.e.  $\gamma^{-1}(0001) = \emptyset$) or a singleton (g.e.  $\gamma^{-1}(29) = \{29\}$).

Each choice of $\mathcal{U}$, $\Sigma$ and $\gamma$ gives rise to an structure $\mathcal{T} := \left(\mathcal{U}, \Sigma, \gamma \right)$ that we will call \emph{arithm\'etique langagi\`ere}. In this structure it is natural to define a notion of proof (see Definition \ref{Defundnudsufbsubfuywwe89}) using the minimal $\sigma$-algebra containing the family of sets $\left( \gamma^{-1}(w) \right)_{w \in \Sigma^{\ast}}$. This notion of proof is a refinement of the ordinary notion of proof in \textbf{ZFC} (see Lemma \ref{lemfundu3498ur9483ur983u9ru39}). In the case of the decimal positional numeration system, considered as an arithm\'etique langagi\`ere, it is easy to write a proof that a positive integer, which is divisible by $10$, it is also divisible by $5$ (just look at the last character).

In this paper we are particularly interested in a family of arithm\'etiques langagi\`eres, denoted $\textbf{KR}_{\lambda}$ and parametrized by a real number $\lambda > 1$. The original motivation for the definition of $\textbf{KR}_{\lambda}$ is that, for $\lambda = 2$, $\gamma(n)$ encodes, up to an injective morphism of monoids, the non-zero coefficients of the polynomials $C_n(q)$, introduced in \cite{kassel2015counting} and \cite{kassel2016complete}. A quickly way to define $C_n(q)$ is as the number of ideals $I$ of the group algebra $\mathbb{F}_q\left[ \mathbb{Z}\oplus \mathbb{Z}\right]$ such that $\mathbb{F}_q\left[ \mathbb{Z}\oplus \mathbb{Z}\right]/I$ is an $n$-dimensional vector space. It is remarkable that these polynomials are related to classical multiplicative functions via modular forms (see \cite{kassel2016fourier}).

We will show that the arithm\'etique langagi\`ere $\textbf{KR}_{2}$ can be used to prove, in a natural way, statements (Theorems \ref{teojosjdiojiosdfjs} and \ref{teo8u8439u98ur39ur93}) concerning semi-perimeters of Pythagorean triangles (Definition \ref{defu48589355hjk5h34hk34}), even-trapezoidal numbers (Definition \ref{defnur3huihuihrwiuh98u89}) and $2$-densely divisible numbers (Definition \ref{defju87fd97s7f7ds987f98sd7}). Also, we will show a statement (Theorem \ref{teojrojiorjrwde}) about generalized middle divisors (Definition \ref{defn897e7eeeee7ew7r7r98789ew7rer}), due to H\"oft \cite{Hoft}, whose proof using our approach involves the whole family of arithm\'etiques langagi\`eres $\left(\textbf{KR}_{\lambda} \right)_{\lambda > 1}$.

\section{Preliminaries}
\subsection{Symmetric Dyck words}
\begin{definition}[Definition 1 in \cite{Caballero1}]
Let $\lambda > 1$ be a real number. For any integer $n \geq 1$ define the word 
$$
\llangle n \rrangle_{\lambda} := w_1 w_2 ... w_k \in \{a,b\}^{\ast},
$$ by means of the expression
$$
w_i := \left\{ \begin{array}{c l}
a & \textrm{if } u_i \in D_n \backslash \left(\lambda D_n\right), \\
b & \textrm{if } u_i \in \left(\lambda D_n\right)\backslash D_n,
\end{array} \right.
$$
where $D_n$ is the set of divisors of $n$, $\lambda D_n := \{\lambda d: \quad d \in D_n\}$ and $u_1, u_2, ..., u_k$ are the elements of the symmetric difference $D_n \triangle \lambda D_n$ written in increasing order.
\end{definition}

\begin{definition}
For each real number $\lambda > 1$ define the language

$$
\mathcal{L}_{\lambda} := \left\{ \llangle n \rrangle_{\lambda}: \quad n \in \mathbb{Z}_{\geq 1} \right\}.
$$
\end{definition}

The \emph{Dyck language}, denoted $\mathcal{D}$, is defined as the $\subseteq$-smallest language over the alphabet $\{a,b\}$ satisfying $\varepsilon \in \mathcal{D}$, $a\mathcal{D}b \subseteq \mathcal{D}$ and $\mathcal{D}\mathcal{D} \subseteq \mathcal{D}$. Words in $\mathcal{D}$ are called \emph{Dyck words}.

The \emph{symmetric Dyck language}, denoted $\mathcal{D}^{\textrm{sym}}$, is defined by
$$
\mathcal{D}^{\textrm{sym}} := \{w \in \mathcal{D}: \quad \widetilde{w} = \sigma\left(w \right) \},
$$
where $\widetilde{w}$ is the mirror image of $w$ and $\sigma: \{a,b\}^{\ast} \longrightarrow \{a,b\}^{\ast}$ is the morphism of monoids given by $a \mapsto b$ and $b \mapsto a$. Words in $\mathcal{D}^{\textrm{sym}}$ are called \emph{symmetric Dyck words}.

\subsection{Irreducible Dyck words}
Let $(\mathcal{D}, \cdot)$ be the monoid of Dyck words endowed with the ordinary concatenation (usually omitted in notation).

It is well-known that $\mathcal{D}$ is freely generated by the language of \emph{irreducible Dyck words} $
\mathcal{D}^{\textrm{irr}} := a\mathcal{D}b
$, i.e. every word in $\mathcal{D}$ may be formed in a unique way by concatenating a sequence of words from $\mathcal{D}^{\textrm{irr}} $. So, there is a unique morphism of monoids $\Omega: (\mathcal{D}, \cdot) \longrightarrow (\mathbb{Z}_{\geq 1},+)$, such that the diagram
$$
\begin{tikzcd}
\mathcal{D} \arrow{r}{} \arrow[swap, dashed]{rd}{\Omega} & \left(\mathcal{D}^{\textrm{irr}}\right)^{\ast} \arrow[two heads]{d}{} \\
 &  \mathbb{Z}_{\geq 1} 
\end{tikzcd} \label{Diagr093jr9j03}
$$
commutes, where $\mathcal{D} \longrightarrow \left(\mathcal{D}^{\textrm{irr}}\right)^{\ast}$ is the identification of $\mathcal{D}$ with the free monoid $\left(\mathcal{D}^{\textrm{irr}}\right)^{\ast}$ and $\left(\mathcal{D}^{\textrm{irr}}\right)^{\ast} \longrightarrow \mathbb{Z}_{\geq 1}$ is just the length of a word in $\left(\mathcal{D}^{\textrm{irr}}\right)^{\ast}$ considering each element of the set $\mathcal{D}^{\textrm{irr}} $ as a single letter (of length $1$). In other words, $\Omega(w)$, with $w \in \mathcal{D}$, is the number of irreducible Dyck words needed to obtain $w$ as a concatenation of them.

\subsection{The central concatenation}

\begin{definition}[from \cite{Caballero3}]
Consider the set $\mathcal{S} := \{aa, ab, ba, bb\}$ endowed with the binary operation, that we will call \emph{central concatenation},
$$
u \triangleleft v  := \varphi^{-1}\left( \varphi(u)\varphi(v)\right),
$$
where $\varphi: \mathcal{S}^{\ast} \longrightarrow \mathcal{S}^{\ast}$ is the bijection given by
\begin{eqnarray*}
\varphi\left( \varepsilon\right) &=& \varepsilon, \\
\varphi\left( x\,u\,y\right) &=& (xy)\,\varphi\left( u\right),
\end{eqnarray*}
for all $x,y \in \{a,b\}$ and $u \in \mathcal{S}^{\ast}$.
\end{definition}

It is easy to check that $\left( \mathcal{S}, \triangleleft \right)$ is a monoid freely generated by $\mathcal{S}$ and having $\varepsilon$ as identity element.

\begin{definition}[from \cite{Caballero3}]
For any $x \in \mathcal{S}$, let 
$$\ell_x: \left(\mathcal{S}^{\ast}, \triangleleft \right) \longrightarrow \left( \mathbb{Z}_{\geq 0}, +\right)$$
be the unique morphism of monoids satisfying
$$
\ell_x(y) := \left\{ \begin{array}{c l}
1 & \textrm{if } x = y, \\
0 & \textrm{if } x \neq y,
\end{array}\right.
$$
for all $y \in \mathcal{S}$.
\end{definition}

It is easy to prove that $\left( \mathcal{D}, \triangleleft\right)$ is a monoid freely generated by $
\mathcal{I} := \mathcal{D}_{\bullet} \backslash \left(\mathcal{D}_{\bullet} \triangleleft\mathcal{D}_{\bullet} \right)
$, where $\mathcal{D}_{\bullet} := \mathcal{D} \backslash\{\varepsilon\}$. The following definition corresponds to the notion of \emph{centered tunnels} introduced for the first time, in an equivalent way, in \cite{Elizalde}.

\begin{definition}[from \cite{Elizalde} and \cite{Caballero3}]
 Let $\ct: \left( \mathcal{D}, \triangleleft\right) \longrightarrow \left( \mathbb{Z}_{\geq 0}, +\right)$ be the  morphism of monoids given by
$$
\ct\left(w\right) := \left\{ \begin{array}{c l}
1 & \textrm{if } w = ab, \\
0 & \textrm{if } w \neq ab,
\end{array}\right.
$$
for all $w \in \mathcal{I}$. We say that $\ct\left(w\right)$ is the \emph{number of centered tunnels} of $w$.
\end{definition}

\section{Logical framework}
\subsection{Th\'eorie langagi\`ere}
Let $\Sigma$ be a finite alphabet. Consider the measurable space $\left(\Sigma^{\ast},  \mathcal{P}\left(\Sigma^{\ast}\right)\right)$ of subsets of $\Sigma^{\ast}$ (languages over the alphabet $\Sigma$), where $\mathcal{P}\left(\Sigma^{\ast}\right)$ is the ordinary $\sigma$-algebra of subsets of $\Sigma^{\ast}$.

\begin{definition}
Let $\mathcal{U}$ be a set. A \emph{th\'eorie langagi\`ere}\footnote{In English we could say \emph{language-theoretic theory}, but it is longer than the French expression.} is a $3$-tuple $\left(\mathcal{U},  \Sigma, \gamma \right)$, where  
$
\gamma: \mathcal{U}\longrightarrow \Sigma^{\ast}
$
 is an application.
\end{definition}

\begin{definition}\label{Defundnudsufbsubfuywwe89}
Let $\mathcal{T} = \left(\mathcal{U},  \Sigma, \gamma \right)$ be a th\'eorie langagi\`ere. Denote by $\mathfrak{U}_{\mathcal{T}}$ the minimal $\sigma$-algebra containing the family of sets $\left( \gamma^{-1}(w) \right)_{w \in \Sigma^{\ast}}$. Given $R, S \in \mathcal{P}\left( \mathcal{U}\right)$, we say that the theorem ``$R=S$'' is \emph{provable in $\mathcal{T}$} if the following statements are provable in \textbf{ZFC},
\begin{enumerate}[label = (\roman*)]
\item ``$R, S \in \mathfrak{U}_{\mathcal{T}}$'',
\item ``$\gamma\left( R\right) = \gamma\left( S\right)$''.
\end{enumerate}
\end{definition}

\begin{lemma}[Fundamental Lemma of Th\'eories Langagi\`eres]\label{lemfundu3498ur9483ur983u9ru39}
Let $\mathcal{T} = \left(\mathcal{U},  \Sigma, \gamma \right)$ be a th\'eorie langagi\`ere.  For all $R, S \in \mathcal{P}\left( \mathcal{U}\right)$, if ``$R=S$'' is provable in $\mathcal{T}$ then ``$R=S$'' is provable in \textbf{ZFC}.
\end{lemma}

\begin{proof}
Suppose that $R, S \in \mathfrak{U}_{\mathcal{T}}$ and  $\gamma\left( R\right) = \gamma\left( S\right)$. 

The statement $R, S \in \mathfrak{U}_{\mathcal{T}}$ and the minimality of $\mathfrak{U}_{\mathcal{T}}$ imply the existence of two languages $L_R, L_S \in \mathcal{P}\left( \Sigma^{\ast}\right)$ such that
$$
R = \bigcup_{w \in L_R} \gamma^{-1}(w) \textrm{ and } S = \bigcup_{w \in L_S} \gamma^{-1}(w).
$$

Without lost of generality we will assume that $\gamma^{-1}(w) \neq \emptyset$ for all $w \in L_R \cup L_S$. It follows that $\gamma(R) = L_R$ and $\gamma(S) = L_S$. The equality $\gamma\left( R\right) = \gamma\left( S\right)$ implies that $L_R = L_S$. Therefore $R = S$.
\end{proof}

A th\'eorie langagi\`ere $\mathcal{T} = \left(\mathcal{U}, \Sigma, \gamma \right)$ satisfying $\mathcal{U} \subseteq \mathbb{Z}_{\geq 1}$ will be called \emph{arithm\'etique langagi\`ere}\footnote{In English we could say \emph{language-theoretic arithmetic}.}.

\begin{definition}
Let $\lambda > 1$ be a real number. Define $\textbf{KR}_{\lambda} := \left(\mathcal{U}, \Sigma, \gamma \right)$, where $\mathcal{U} := \mathbb{Z}_{\geq 1}$, $\Sigma := \{a, b\}$ and $\gamma(n) := \llangle n \rrangle_{\lambda}$.
\end{definition}

\section{Middle divisors}

Let $C_n(q)$ be the polynomial mentioned in the introduction. It was proved in \cite{kassel2016complete} that $C_n(q) = (q-1)^2 P_n(q)$, for some polynomial $P_n(q)$ whose coefficients are non-negative integers.

Divisors $d|n$ satisfying $\sqrt{n/2} < d \leq \sqrt{2n}$ are called \emph{middle divisors} of $n$. These divisors were studied in \cite{kassel2016complete}, \cite{Hoft} and \cite{Vatne}. The coefficient of $q^{n-1}$ in $P_n(q)$, denoted $a_{n,0}$, counts the number of middle divisors of $n$. The following definition provides a generalization of the arithmetical function $
a_{n,0}$.

\begin{definition}[from \cite{Caballero3}]\label{defn897e7eeeee7ew7r7r98789ew7rer}
Consider a real number $\lambda > 1$. Let $n \geq 1$ be an integer. The number of \emph{$\lambda$-middle divisors} of $n$, denoted $\textrm{middle}_{\lambda}(n)$, is the number of divisors $d$ of $n$ satisfying
$$
\sqrt{\frac{n}{\lambda}} < d \leq \sqrt{\lambda n}.
$$
\end{definition}

 A \emph{block polynomial} is a polynomial of the form $B(q) = q^i + q^{i+1} + q^{i+2} + ... + q^j$, with $0 \leq i < j$. The smallest number $k$ of block polynomials $B_1(q), B_2(q), ..., B_k(q)$ such that

$$
P_n(q) = \alpha_1 B_1(q) + \alpha_2 B_2(q) + ... + \alpha_k B_k(q),
$$ 
for some $\alpha_1, \alpha_2, ..., \alpha_k \in \mathbb{Z}$, will be called the \emph{number of blocks} of $n$ and denoted $\blocks(n) := k$. The arithmetical function $\blocks(n) $ is generalized in the following definition.

\begin{definition}[from\footnote{In \cite{Caballero2}, the function $\blocks_{\lambda}(n)$ is called the \emph{number of connected components} of $\mathcal{T}_{\lambda}(n)$.} \cite{Caballero2}]
Consider a real number $\lambda > 1$. Let $n \geq 1$ be an integer. We define the \emph{number of $\lambda$-blocks} of $n$, denoted $\blocks_{\lambda}(n)$, as the number of connected components of
$$
\bigcup_{d|n} \left[d, \lambda d\right].
$$
\end{definition}

Theorem 3 in \cite{Hoft} (we call it \emph{H\"oft's theorem}) states the equivalent between $\textrm{middle}_{2}(n) > 0$ and $\blocks_{2}(n) \equiv 1 \pmod{2}$, for any integer $n \geq 1$. The following result is a generalization of H\"oft's original result.

\begin{theorem}[Generalized H\"oft's theorem]\label{teojrojiorjrwde}
Let $\lambda > 1$ be a real number. For each integer $n \geq 1$, we have that $\textrm{middle}_{\lambda}(n) > 0$ if and only if $\blocks_{\lambda}(n)$ is odd. Furthermore, this theorem is provable in $\textbf{KR}_{\lambda}$.
\end{theorem}

H\"oft's proof in \cite{Hoft} follows the general lines of traditional proofs in Elementary Number Theory. Our proof of Theorem \ref{teojrojiorjrwde} will be based on properties of Dyck words.  We will use the following auxiliary results.

\begin{lemma}\label{lempkjsljdf90sf90s8}
 For any integer $n \geq 1$ and any real number $\lambda > 1$, we have that $\llangle n \rrangle_{\lambda}$ is a symmetric Dyck word, i.e. $\mathcal{L}_{\lambda} \subseteq \mathcal{D}^{\textrm{sym}}$.
\end{lemma}

\begin{proof}
See Theorem 2(i) in \cite{Caballero1}.
\end{proof}

\begin{lemma}\label{propjfsdfuisdhfiuhishifuhdhidshf}
Let $\lambda > 1$ be a real number. For any integer $n \geq 1$,  $\ct\left(\llangle n \rrangle_{\lambda}\right) = \textrm{middle}_{2}(n)$.
\end{lemma}

\begin{proof}
See Lemma 3.7 in \cite{Caballero3}.
\end{proof}

\begin{lemma}\label{propmmknnjknjknnbvvgcgcf}
Let $\lambda > 1$ be a real number. 
For any integer $n \geq 1$, $\Omega\left(\llangle n \rrangle_{\lambda}\right) = \blocks_{\lambda}(n)$.
\end{lemma}

\begin{proof}
See Theorem 2 in \cite{Caballero2}.
\end{proof}

\begin{lemma}\label{lemjjfuishfishfishfishduh}
Consider the languages over the alphabet $\{a,b\}$,
\begin{eqnarray*}
L_R &:=& \left\{w \in \mathcal{D}^{\textrm{sym}}: \quad \ct(w) > 0 \right\}, \\
L_S &:=& \left\{w \in \mathcal{D}^{\textrm{sym}}: \quad \Omega(w) \textrm{ odd} \right\}.
\end{eqnarray*}
We have that $L_R = L_S$.
\end{lemma}

\begin{proof}
Take $w \in L_S$. By definition of $L_S$, we have that $\Omega(w)$ is odd. By Lemma \ref{lempkjsljdf90sf90s8}, there are $u, v \in \mathcal{D}$ such that $w = u \, v \, \sigma\left( \widetilde{u}\right)$ and $v$ is irreducible. By definition of $\mathcal{D}^{\textrm{irr}}$, there is  $v^{\prime} \in \mathcal{D}$ satisfying $v = a v^{\prime} b$. So, $w = u \, a \, v^{\prime} \, b \, \sigma\left( \widetilde{u}\right)$. It follows that $\ct(w) > 0$. Hence $w \in L_R$.

Now, take $w \in L_R$. By definition of $L_R$ we have that $\ct(w) > 0$. By Lemma \ref{lempkjsljdf90sf90s8}, there are $u, v^{\prime} \in \mathcal{D}$ such that $w = u \, a \, v^{\prime} \, b \, \sigma\left( \widetilde{u}\right)$. The Dyck word $v := a v^{\prime} b$ is irreducible and $w = u \, v \, \sigma\left( \widetilde{u}\right)$. It follows that $\Omega(w) = 1 + 2 \Omega(u)$. Hence, $w \in L_S$.

Therefore, $L_R = L_S$.
\end{proof}

\begin{proof}[Proof of Theorem \ref{teojrojiorjrwde}]
Consider a fixed real number $\lambda > 1$. Define the sets
\begin{eqnarray*}
R &:=& \left\{n \in \mathbb{Z}_{\geq 1}: \quad \textrm{middle}_{\lambda}(n) > 0 \right\}, \\
S &:=& \left\{n \in \mathbb{Z}_{\geq 1}: \quad \blocks_{\lambda}(n) \textrm{ odd} \right\}.
\end{eqnarray*}

Let $L_R$ and $L_S$ be the languages defined in Lemma \ref{lemjjfuishfishfishfishduh}. In virtue of Lemmas \ref{propjfsdfuisdhfiuhishifuhdhidshf} and \ref{propmmknnjknjknnbvvgcgcf},
$$
R = \bigcup_{w \in L_R} \gamma^{-1}(w)\in \mathfrak{U}_{\textbf{KR}_{\lambda}} \textrm{ and } S = \bigcup_{w \in L_S} \gamma^{-1}(w)\in \mathfrak{U}_{\textbf{KR}_{\lambda}},
$$
where $\gamma$ is from $\textbf{KR}_{\lambda} = \left(\mathcal{U}, \Sigma, \gamma \right)$. By definition of $\mathcal{L}_{\lambda}$, it follows that $\gamma\left( R\right) = L_R \cap \mathcal{L}_{\lambda}$ and $\gamma\left( S\right) = L_S \cap \mathcal{L}_{\lambda}$. In virtue of Lemma \ref{lemjjfuishfishfishfishduh}, $L_R = L_S$. So, $\gamma\left( R\right) = \gamma\left( S\right)$. By Definition \ref{Defundnudsufbsubfuywwe89}, ``$R = S$'' is provable in $\textbf{KR}_{\lambda}$. Using Lemma \ref{lemfundu3498ur9483ur983u9ru39}, we conclude that $R = S$.
\end{proof}

\section{Semi-perimeters of Pythagorean triangles}

\begin{definition}\label{defu48589355hjk5h34hk34}
Let $n \geq 1$ be an integer. We says that $n$ \emph{is the semi-perimeter of a Pythagorean triangle} if there are three integers $x, y, z \in \mathbb{Z}_{\geq 1}$ satisfying
$$
x^2 + y^2 = z^2 \textrm{ and } \frac{x+y+z}{2} = n. 
$$
\end{definition}

In order to work with semi-perimeters of Pythagorean triangles, we will need the following language-theoretical characterization.

\begin{lemma}\label{Lem89ru34899r834ur9}
An integer $n \geq 1$ is not the semi-perimeter of a Pythagorean triangle if and only if $\llangle n \rrangle_{2} \in \left( ab\right)^{\ast}$. 
\end{lemma}

We will use the following auxiliary result.
\begin{lemma}\label{propklfjdkls98798798fsdsf}
 For any integer $n \geq 1$ and any real number $\lambda > 1$, the height of the Dyck path $\llangle n \rrangle_{\lambda}$ is the largest value of $h$ such that we can find $h$ divisors of $n$, denoted $d_1, d_2, ..., d_h$, satisfying
$$
d_1 < d_2 < ... < d_h < \lambda d_1.
$$
\end{lemma}

\begin{proof}
See Theorem 2(ii) in \cite{Caballero1}.
\end{proof}

\begin{proof}[Proof of Lemma \ref{Lem89ru34899r834ur9}]
From the explicit formula for Pythagorean triples (see \cite{sierpinski2003pythagorean}), it follows in a straightforward way that an integer $n \geq 1$ is the semi-perimeter of a Pythagorean triangle if and only if there are two divisors of $n$, denoted $d_1$ and $d_2$, satisfying,
$$
d_1 < d_2 < 2d_1.
$$

By Lemma \ref{lempkjsljdf90sf90s8}, $\llangle n \rrangle_{2}$ is a Dyck word, so its height as Dyck path is well-defined. In virtue of Lemma \ref{propklfjdkls98798798fsdsf}, such divisors $d_1$ and $d_2$ do exist if and only if the height of $\llangle n \rrangle_{2}$ at least $2$. Therefore, $n$ is not the semi-perimeter of a Pythagorean triangle if and only if $\llangle n \rrangle_{2} \in \left( ab\right)^{\ast}$.
\end{proof}

\subsection{Even-trapezoidal numbers}
The number of partitions of a given integer $n \geq 1$ into an even number of consecutive parts was study in \cite{hirschhorn2009partitions}.

\begin{definition}\label{defnur3huihuihrwiuh98u89}
Let $n \geq 1$ be an integer. We says that $n$ \emph{even-trapezoidal} if there is at least a partition of $n$ into an even number of consecutive parts, i.e.
$$
n = \sum_{k=0}^{2m-1} \left(a + k\right)
$$
for two integers $a \geq 1$ and $m \geq 1$.
\end{definition}

It is rather easy to check that a power of $2$ is neither even-trapezoidal nor the semi-perimeter of a Pythagorean triangle. Nevertheless, the converse statement is non-trivial.

\begin{theorem}\label{teojosjdiojiosdfjs}
Let $n \geq 1$ be an integer. We have that $n$ is a power of $2$ (including $n = 1$) if and only if $n$ is neither even-trapezoidal nor the semi-perimeter of a Pythagorean triangle. Furthermore, this theorem is provable in $\textbf{KR}_{2}$.
\end{theorem}

We will use the following auxiliary results.

\begin{lemma}\label{Lemj89jw9efw98j98je9}
For all integers $n \geq 1$, we have that $n$ is a power of $2$ (including $n = 1$) if and only if $\llangle n \rrangle_2 = ab$.
\end{lemma}

\begin{proof}
Take $n \in \mathbb{Z}_{\geq 1}$. 

Suppose that $\llangle n \rrangle_2 = ab$. By definition of $\llangle n \rrangle_2$, the length of $\llangle n \rrangle_2$ is two times the number of odd divisors of $n$. So, $n$ has exactly $1$ odd divisors. It follows that $n$ is a power of $2$ (including $n = 1$).

Suppose that $n$ is a power of $2$ (including $n = 1$). It follows that 
$$
D_n \triangle 2D_n = \left\{ 1 < 2n\right\},
$$
with $1 \in D_n \backslash \left( 2D_n \right)$ and $2n \in \left( 2D_n \right) \backslash D_n$. By definition of $\llangle n \rrangle_2$, we conclude that $\llangle n \rrangle_2 = ab$.
\end{proof}

\begin{lemma}\label{lemjsiojdoifjosdjfofjs79}
 For any integer $n \geq 1$ and any real number $\lambda > 1$, we have
$$
\ell_{ab}\left(\llangle n \rrangle_{\lambda} \right) = \# \left\{ d | n: \quad d \not\in \lambda D_n \textrm{ and } d < \sqrt{\lambda n}\right\},
$$
where $D_n$ is the set of divisors of $n$.
\end{lemma}

\begin{proof}
See Lemma 3.4. in \cite{Caballero3}.
\end{proof}

\begin{lemma}\label{Lemksjdlkjfkljsfl}
For all $n \geq 1$, we have that $n$ is not even-trapezoidal if and only if $\llangle n \rrangle_2 \in \left\{ a^k \, b^k: \quad k \in \mathbb{Z}_{\geq 1} \right\}$.
\end{lemma}

\begin{proof}
It was proved in \cite{hirschhorn2009partitions} that the number of partitions of $n$ into an even number of consecutive parts is precisely the cardinality of the set
$$
\left\{ d | n: \quad d \not\in 2 D_n \textrm{ and } d > \sqrt{2 n}\right\}.
$$

Notice that if  $d = \sqrt{2 n}$ is a divisor of $n$, then $d = 2\frac{n}{d}$ is even. So, an integer $n \geq 1$ is not even-trapezoidal if and only if 
$$
\# \left\{ d | n: \quad d \not\in 2 D_n \textrm{ and } d < \sqrt{2 n}\right\} = \frac{1}{2}\,\left|\llangle n \rrangle_2 \right|.
$$

By Lemma \ref{lempkjsljdf90sf90s8}, $\llangle n \rrangle_{2}$ is a Dyck word, so $\ell_{ab}\left(\llangle n \rrangle_{2} \right)$ is well-defined. In virtue of Lemma \ref{lemjsiojdoifjosdjfofjs79}, an integer $n \geq 1$ is not even-trapezoidal if and only if 
$$
\ell_{ab}\left(\llangle n \rrangle_{2} \right) = \frac{1}{2}\,\left|\llangle n \rrangle_2 \right|.
$$

This last condition holds if and only if there is $k \in \mathbb{Z}_{\geq 1}$ such that $\llangle n \rrangle_{2} = a^k \, b^k$, because $\llangle n \rrangle_{2}$ is a Dyck word. Therefore, $n$ is not even-trapezoidal if and only if $\llangle n \rrangle_2 \in \left\{ a^k \, b^k: \quad k \in \mathbb{Z}_{\geq 1} \right\}$.
\end{proof}

\begin{proof}[Proof of Theorem \ref{teojosjdiojiosdfjs}]
Define the sets
\begin{eqnarray*}
R &:=& \left\{2^m: \quad m \in \mathbb{Z}_{\geq 0} \right\}, \\
S &:=& \left\{n \in \mathbb{Z}_{\geq 1}: \quad \begin{array}{c}
\neg\left(n \textrm{ even-trapezoidal}\right) \textrm{ and} \\
\neg\left(n \textrm{ semi-perimeter of a Pythagorean triangle}\right)
\end{array} \right\}.
\end{eqnarray*}

Consider the languages
\begin{eqnarray*}
L_R &=& \left\{ ab \right\}, \\
L_S &=& \left\{ a^k \, b^k: \quad k \in \mathbb{Z}_{\geq 1} \right\} \cap \left(ab \right)^{\ast}.
\end{eqnarray*}

 In virtue of Lemmas \ref{Lemj89jw9efw98j98je9}, \ref{Lem89ru34899r834ur9} and \ref{Lemksjdlkjfkljsfl},
$$
R = \bigcup_{w \in L_R} \gamma^{-1}(w)\in \mathfrak{U}_{\textbf{KR}_{2}} \textrm{ and } S = \bigcup_{w \in L_S} \gamma^{-1}(w)\in \mathfrak{U}_{\textbf{KR}_{2}},
$$
where $\gamma$ is from $\textbf{KR}_{\lambda} = \left(\mathcal{U}, \Sigma, \gamma \right)$. Furthermore, $\gamma\left( R\right) = L_R \cap \mathcal{L}_{2}$ and $\gamma\left( S\right) = L_S \cap \mathcal{L}_{2}$.

It easily follows that $L_R = L_S$. So, $\gamma\left( R\right) = \gamma\left( S\right)$. By Definition \ref{Defundnudsufbsubfuywwe89}, ``$R = S$'' is provable in $\textbf{KR}_{2}$. Using Lemma \ref{lemfundu3498ur9483ur983u9ru39}, we conclude that $R = S$.
\end{proof}

\subsection{Densely divisible numbers}

The so-called $\lambda$-densely divisible numbers were introduced in \cite{polymath8} by the project \emph{polymath8}, led by Terence Tao.

\begin{definition}\label{defju87fd97s7f7ds987f98sd7}
Consider a real number $\lambda > 1$. Let $n \geq 1$ be an integer. We say that $n$ is \emph{$\lambda$-densely divisible} if $\blocks_{\lambda}(n) = 1$.
\end{definition}

Again, it can be proved in a straightforward way that powers of $2$ are $2$-densely divisible number. But it is more complicated to prove that, for a given positive integer, to be a $2$-densely divisible number,
without being the semi-perimeter of a Pythagorean triangle, it is enough to be a power of $2$.

\begin{theorem}\label{teo8u8439u98ur39ur93}
Let $n \geq 1$ be an integer. We have that $n$ is a power of $2$ (including $n = 1$) if and only if both $n$ is $2$-densely divisible and it is not the semi-perimeter of a Pythagorean triangle. Furthermore, this theorem is provable in $\textbf{KR}_{2}$.
\end{theorem}

We will use the following auxiliary results.
\begin{lemma}\label{propiojsirjsdoijfdoisjfdo}
Let $\lambda > 1$ be a real number. 
For any integer $n \geq 1$, we have that $\llangle n \rrangle_{\lambda}$ is irreducible (i.e. $\llangle n \rrangle_{\lambda} \in \mathcal{D}^{irr}$) if and only if $n$ is $\lambda$-densely divisible.
\end{lemma}

\begin{proof}
It is the case corresponding to $\blocks(n) = 1$ in Lemma \ref{propmmknnjknjknnbvvgcgcf}.
\end{proof}

\begin{proof}[Proof of Theorem \ref{teo8u8439u98ur39ur93}]
Define the sets
\begin{eqnarray*}
R &:=& \left\{2^m: \quad m \in \mathbb{Z}_{\geq 0} \right\}, \\
S &:=& \left\{n \in \mathbb{Z}_{\geq 1}: \quad \begin{array}{c}
\left(n \textrm{ 2-densely divisible}\right) \textrm{ and} \\
\neg\left(n \textrm{ semi-perimeter of Pythagorean triangle}\right)
\end{array} \right\}.
\end{eqnarray*}

Consider the languages
\begin{eqnarray*}
L_R &=& \left\{ ab \right\}, \\
L_S &=& \left\{ w \in \mathcal{D}: \quad w \textrm{ irreducible} \right\} \cap \left(ab \right)^{\ast}.
\end{eqnarray*}

 In virtue of  Lemma \ref{Lemj89jw9efw98j98je9}, \ref{Lem89ru34899r834ur9} and \ref{propiojsirjsdoijfdoisjfdo},
$$
R = \bigcup_{w \in L_R} \gamma^{-1}(w)\in \mathfrak{U}_{\textbf{KR}_{2}} \textrm{ and } S = \bigcup_{w \in L_S} \gamma^{-1}(w)\in \mathfrak{U}_{\textbf{KR}_{2}},
$$
where $\gamma$ is from $\textbf{KR}_{2} = \left(\mathcal{U}, \Sigma, \gamma \right)$. Furthermore, $\gamma\left( R\right) = L_R \cap \mathcal{L}_{2}$ and $\gamma\left( S\right) = L_S \cap \mathcal{L}_{2}$.

It easily follows that $L_R = L_S$. So, $\gamma\left( R\right) = \gamma\left( S\right)$. By Definition \ref{Defundnudsufbsubfuywwe89}, ``$R = S$'' is provable in $\textbf{KR}_{2}$. Using Lemma \ref{lemfundu3498ur9483ur983u9ru39}, we conclude that $R = S$.
\end{proof}

\section{Conclusions}

In this paper we showed that some non-trivial elementary number-theoretical theorems are susceptible to be transformed into relationships among formal languages and then proved by rather trivial arguments from Language Theory.

\section*{Acknowledge}
The author thanks S. Brlek and  C. Reutenauer for they valuable comments and suggestions concerning this research.


\end{document}